\documentclass[11pt]{article}
\date{}
\usepackage{amssymb}
\usepackage{amsmath}
\usepackage{amsthm}
\usepackage{graphicx}
\usepackage{indentfirst}
\allowdisplaybreaks[4]
\newtheorem{theorem}{Theorem}

\newtheorem{definition}[theorem]{Definition}

\newtheorem{lemma}[theorem]{Lemma}

\numberwithin{equation}{section}
\numberwithin{theorem}{section}
\newcommand{\keywords}[1]{\par\noindent\textbf{Keywords:} #1}
\newcommand{\subjclass}[2]{
  \par\noindent\textbf{Mathematics Subject Classification} #2}

\begin{document}

\title{Liouville theorem for the inequality $\Delta_m u+f(u)\leq 0$ on Riemannian manifolds }

\author{ Biqiang Zhao\footnote{
		E-mail addresses: 2306394354@pku.edu.cn}}
            
\maketitle
\setlength{\parindent}{2em}

\begin{abstract}
%% Text of abstract
In this paper, we study the quasilinear inequality $ \Delta_m u+f(u)\leq 0$ on a complete Riemannian manifold, where
\begin{align*}
    m>1,\alpha>m-1 \quad and \quad f(t)> 0,\alpha f(t)-tf^{'}(t)\geq 0, \forall t>0.
\end{align*}
If for some point $x_0$ and large enough $r$,
 \begin{align*}
     vol B_r(x_0)\leq C r^p ln^q r,
 \end{align*}
where $p=\frac{m\alpha}{\alpha-(m-1)},q=\frac{m-1}{\alpha-(m-1)}$ and $B_r(x_0) $ is a geodesic ball of radius $r$ centered at $x_0$, then the inequality possesses no positive weak solution. This generalizes the result in \cite{AS,Sun}.

\end{abstract}

\subjclass [{ 35J61 · 58J05}

%% \linenumbers

%% main text
\section{Introduction}
\label{1}
    In this paper, we are concerned with positive solutions of the quasilinear inequality
    \begin{align}
       \Delta_m u+f(u)\leq 0
    \end{align}
     on a geodesically complete connected Riemannian manifold. The operator $ \Delta_m u=div(|\nabla u|^{m-2}\nabla u) $ is called $m$-Laplacian. 
     \par
      In the seminal paper \cite{GS}, Gidas and Spruck proved that if $ n>2$ and $ 1<\alpha<\frac{n+2}{n-2}$, the Lane-Emden equation
     \begin{align*}
         \Delta u+u^\alpha=0\quad in\ \mathbb{R}^n
     \end{align*}
     has no positive solutions. Consider the inequality 
     \begin{align*}
         \Delta u+u^\alpha\leq 0\quad in\ \mathbb{R}^n.
     \end{align*}
     It is known that there is no positive solution if $1\leq \alpha \leq \frac{n}{n-2}$ (see \cite{CM,NS1,NS2,GS}). Later, Serrin and Zou \cite{SZ} generalized the Gidas and Spruck's result to quasilinear elliptic equation. Since then many improvements or generalizations have been obtained for elliptic equations and inequalities (see \cite{CDM,CMP,EP1,EP2}).
     \par
      Next, let's return to results on Riemannian manifolds. Cheng and Yau \cite{CY} proved that if $vol B_r(x_0)\leq Cr^2$ for all large enough $r$, then any positive solution of $\Delta u\leq 0$ is constant. This idea was developed by Grigor'yan and Sun in \cite{AS}. They proved the inequality 
      \begin{align*}
          \Delta u+u^\alpha\leq 0\quad in\ M
      \end{align*}
      has no positive weak solution under the volume growth condition:
      \begin{align*}
          vol B_r(x_0)\leq Cr^{\frac{2\alpha}{\alpha-1}}ln^{\frac{1}{\alpha-1}} r
      \end{align*}
      for $r$ large enough. The analogous results have also been generalized for the $m-$Laplacian operator. In \cite{SXX}, Sun, Xiao and Xu derived Liouville theorems for inequality $\Delta_m u+u^p |\nabla u|^q \leq 0 $ for $(p,q)\in (-\infty,\infty)\times (-\infty,\infty)$. In \cite{PDF}, Mastrolia, Monticelli and Punzo investigated the nonexistence of solutions to parabolic differential inequalities with a potential. Further results for the $m-$Laplacian operator are also achieved when $M$ is a Riemannian manifold (see \cite{GSXX,GSV,HWW,DZ,PDF2}). 
      \par 
       In the present paper, our purpose is to investigate the nonexistence of positive solutions of inequality (1.1). 
        \par
        A function $f\in C^0([0,\infty))\cap C^1((0,\infty))$ is called subcritical with exponent $\alpha$ if 
        \begin{align*}
            f(t)> 0,\quad \alpha f(t)-tf^{'}(t)\geq 0,\  \forall t>0.
        \end{align*}
        Clearly, the following functions are subcritical with exponent $\alpha$
        \begin{align*}
               \quad f(t)=t^\alpha, \quad  f(t)=t^\alpha+t^\beta, \quad  f(t)=\frac{t^\alpha}{1+t^\beta},
        \end{align*}
        where $\alpha>\beta>0 $.
        
        \par 
        Our main result is the following theorem.
        \begin{theorem}
            Let $M$ be a connected geodesically complete Riemannian manifold. Assume $f$ is subcritical with exponent $\alpha>m-1$ and set
        \begin{align*}
            p=\frac{m\alpha}{\alpha-(m-1)},\quad q=\frac{m-1}{\alpha-(m-1)}.
        \end{align*}
        If for some $x_0\in M$ and $C>0$,
            \begin{align*}
                vol B_r(x_0)\leq Cr^p ln^q r
            \end{align*}
            holds for large enough $r$, then (1.1) has no positive weak solution.
        \end{theorem}
        \par
        Since the constant $C$ is not important, it may vary at different occurrences.

        \section{Proof of Theorem 1.1}
        In order to prove Theorem 1.1, we shall introduce some notations. Set 
       \begin{align*}
           W_{loc}^{1,m}(M)=\{f:M\xrightarrow{} \mathbb{R}|f\in L_{loc}^m,\nabla f\in L_{loc}^m\}.
       \end{align*}
        Denote by $ W_{c}^{1,m}(M)$ the subspace of $W_{loc}^{1,m}(M) $ of functions with compact support.
        \begin{definition}
            Let $M$ be a connected geodesically complete Riemannian manifold. Assume $f$ is subcritical with exponent $\alpha>m-1$. A function $u>0$ is called a positive weak solution of the inequality (1.1) if $u\in W_{loc}^{1,m}(M)\cap L^{\infty}_{loc}(M),\ \frac{1}{u}\in L^{\infty}_{loc}(M)$ and for any nonnegative function $\psi\in W_{c}^m(M)$, the following inequality holds
            \begin{align}
                -\int_M|\nabla u|^{m-2}\langle \nabla u, \nabla \psi \rangle d\mu +\int_M f\psi d\mu \leq 0,
            \end{align}
            where $\langle, \rangle $ is the inner product.
        \end{definition}
        
        Before proving Theorem 1.1, we establish the following lemma.
        \begin{lemma}
            Let $s>\frac{m\alpha}{\alpha-(m-1)(a+1)}$ for every $a\in (0,min\{\frac{\alpha}{m-1}-1,m-1\})$. If $u$ is a positive weak solution of (1.1), then there exists a constant $C>0$ such that for every function $ 0\leq \varphi\leq 1$ with $\varphi \in W_{c}^{1,m}(M)$ one has
            \begin{align}
                &\int_M u^{-\frac{(m-1)\alpha}{\alpha-m+1}}f^{\frac{\alpha}{\alpha-m+1}}\varphi^s
                \nonumber\\
                \leq&  C a^{-(m-1)-\frac{(m-1-a)(m-1)^2}{m(\alpha-m+1)}}\left(\int_M |\nabla \varphi|^{\frac{m(\alpha-a)}{\alpha-m+1}} \right)^{\frac{m-1}{m}}  \nonumber
                \\
                &\cdot  \left(\int_{M\backslash K} u^{-\frac{(m-1)\alpha}{\alpha-m+1}}f^{\frac{\alpha}{\alpha-m+1}}\varphi^s\right)^{\frac{(m-1)(a+1)}{m\alpha}}    \left(\int_M |\nabla \varphi|^{\frac{m\alpha}{\alpha-(m-1)(a+1)}} \right)^{\frac{\alpha-(m-1)(a+1)}{m\alpha}}           
            \end{align}
            and
            \begin{align}
                &\left(\int_M u^{-\frac{(m-1)\alpha}{\alpha-m+1}}f^{\frac{\alpha}{\alpha-m+1}}\varphi^s\right)^{1-\frac{(m-1)(a+1)}{m\alpha}} \nonumber
                \\
                \leq & C a^{-(m-1)-\frac{(m-1-a)(m-1)^2}{m(\alpha-m+1)}}\left(\int_M |\nabla \varphi|^{\frac{m(\alpha-a)}{\alpha-m+1}}\right)^{\frac{m-1}{m}}  \nonumber
                \\
                &\cdot \left(\int_M |\nabla \varphi|^{\frac{m\alpha}{\alpha-(m-1)(a+1)}} \right)^{\frac{\alpha-(m-1)(a+1)}{m\alpha}},
            \end{align}
            where $K=\{x\in M| \varphi(x)=1\}$.
        \end{lemma}
        \begin{proof}
            Since (2.3) is a direct consequence of (2.2), we only prove (2.2). Take the test function $\psi=u^{-a}(u^{-\alpha}f)^b\varphi^s=u^{-a-b\alpha}f^b\varphi^s$, where $b>0$ is a constant to be determined later. Then we have
            \begin{align*}
                \nabla \psi=-(a+b\alpha)u^{-a-b\alpha-1}f^b\varphi^s\nabla u+bu^{-a-b\alpha}f^{b-1}f^{'}\varphi^s\nabla +su^{-a-b\alpha}f^b\varphi^{s-1}\nabla \varphi.
            \end{align*}
            From (2.1), we have
            \begin{align*}
                &(a+b\alpha)\int_M |\nabla u|^m u^{-a-b\alpha-1}f^b\varphi^s d\mu-b\int_M |\nabla u|^m u^{-a-b\alpha}f^{b-1}f^{'}\varphi^s d\mu
                \\
                \leq& s\int_M |\nabla u|^{m-2} u^{-a-b\alpha}f^b\varphi^{s-1}\langle \nabla u,\nabla  \varphi\rangle d\mu-\int_Mu^{-a-b\alpha}f^{b+1}\varphi^s.
            \end{align*}
        Since $f$ is subcritical with exponent $\alpha$, we obtain
        \begin{align}
           &a \int_M |\nabla u|^m u^{-a-b\alpha-1}f^b\varphi^s d\mu+\int_Mu^{-a-b\alpha}f^{b+1}\varphi^s d\mu
           \nonumber
           \\
           \leq& s\int_M |\nabla u|^{m-2} u^{-a-b\alpha}f^b\varphi^{s-1}\langle \nabla u,\nabla  \varphi\rangle d\mu.
        \end{align}
        Applying Young's inequality to the right hand side of (2.4), we have
        \begin{align*}
          & s \int_M |\nabla u|^{m-2} u^{-a-b\alpha}f^b\varphi^{s-1}\langle \nabla u,\nabla  \varphi\rangle d\mu 
          \\
          \leq& s \int_M |\nabla u|^{m-1} u^{-a-b\alpha}f^b\varphi^{s-1}|\nabla  \varphi| d\mu 
        \\
        \leq & \int_M (u^{-\frac{(a+b\alpha+1)(m-1)}{m}}f^{\frac{m-1}{m}b}|\nabla u|^{m-1}a^{\frac{m-1}{m}}\varphi^{\frac{m-1}{m}s})
        \\
        &\cdot (a^{-\frac{m-1}{m}}u^{\frac{(a+b\alpha+1)(m-1)}{m}-a-b\alpha}f^{\frac{m}{b}}\varphi^{\frac{s}{m}-1}s|\nabla \varphi| )d\mu
        \\
        \leq &\frac{1}{2}a \int_M |\nabla u|^m u^{-a-b\alpha-1}f^b\varphi^s d\mu+\frac{C}{a^{m-1}}\int_M u^{-a-b\alpha+m-1}f^b\varphi^{s-m}|\nabla\varphi|^m d\mu.
        \end{align*}
        By choosing $b=\frac{m-1-a}{\alpha-m+1}$ and using Young's inequality, we have
        \begin{align}
            &\frac{a}{2}\int_M |\nabla u|^m u^{-a-b\alpha-1}f^b\varphi^s d\mu d\mu+\int_Mu^{-a-b\alpha}f^{b+1}\varphi^sd\mu
           \nonumber
           \\
           \leq& \frac{C}{a^{m-1}}\int_M u^{-a-b\alpha+m-1}f^b\varphi^{s-m} |\nabla\varphi|^m d\mu \nonumber
           \\
           =& \frac{C}{a^{m-1}}\int_M u^{-\frac{(m-1)(m-1-a)}{\alpha-m+1}}f^{\frac{m-1-a}{\alpha-m+1}}\varphi^{s-m} |\nabla\varphi|^m d\mu
           \\
           \leq & \int_M u^{-\frac{(m-1)(m-1-a)}{\alpha-m+1}}f^{\frac{m-1-a}{\alpha-m+1}}\varphi^{\frac{m-1-a}{\alpha-a}s}\frac{C}{a^{m-1}}\varphi^{\frac{\alpha-m+1}{\alpha-a}s-m}|\nabla \varphi|^md\mu \nonumber
           \\
           \leq & \frac{1}{2}\int_Mu^{-\frac{(m-1)(\alpha-a)}{\alpha-m+1}}f^{\frac{\alpha-a}{\alpha-m+1}}\varphi^sd\mu \nonumber
           \\
           &+Ca^{-\frac{(m-1)(\alpha-a)}{\alpha-m+1}}\int_M\varphi^{s-\frac{m(\alpha-a)}{\alpha-m+1}}|\nabla \varphi|^{\frac{m(\alpha-a)}{\alpha-m+1}}d\mu. \nonumber
        \end{align}
         Thus we conclude that 
        \begin{align}
            &\frac{a}{2}\int_M |\nabla u|^m u^{-a-b\alpha-1}f^b\varphi^s d\mu d\mu+\frac{1}{2}\int_Mu^{-a-b\alpha}f^{b+1}\varphi^sd\mu
           \nonumber
           \\
           \leq&Ca^{-\frac{(m-1)(\alpha-a)}{\alpha-m+1}}\int_M\varphi^{s-\frac{m(\alpha-a)}{\alpha-m+1}}|\nabla \varphi|^{\frac{m(\alpha-a)}{\alpha-m+1}}d\mu 
           \\
           \leq&Ca^{-\frac{(m-1)(\alpha-a)}{\alpha-m+1}}\int_M|\nabla \varphi|^{\frac{m(\alpha-a)}{\alpha-m+1}}d\mu, \nonumber
        \end{align}
        since $0\leq \varphi\leq 1$ and
        \begin{align*}
            s>\frac{m\alpha}{\alpha-(m-1)(a+1)}>\frac{m(\alpha-a)}{\alpha-m+1}.
        \end{align*}
        Choosing another test function $\psi=(u^{-\alpha} f)^t \varphi^s$ for some $t>0$, we obtain
        \begin{align*}
            &\int_M u^{-\alpha t}f^{t+1}\varphi^sd\mu \leq s\int_M |\nabla u|^{m-1}u^{-\alpha t}f^t\varphi^{s-1}|\nabla \varphi|d\mu                   \\          
            = & s\int_M|\nabla u|^{m-1}u^{-\frac{(m-1)(a+b\alpha+1)}{m}}f^{\frac{m-1}{m}b}\varphi^{\frac{m-1}{m}s} 
            \\
            &\cdot u^{\frac{(m-1)(a+b\alpha+1)}{m}-\alpha t}f^{t-\frac{m-1}{m}b}\varphi^{\frac{s}{m}-1}|\nabla \varphi|d\mu
            \\
            \leq & C\left(\int_M |\nabla u|^mu^{-a-b\alpha-1}f^b\varphi^s d\mu\right)^{\frac{m-1}{m}}
            \\
            &\cdot \left(\int_Mu^{(m-1)(a+b\alpha+1)-m\alpha t}f^{mt-(m-1)b}\varphi^{s-m} |\nabla \varphi|^m d\mu\right)^{\frac{1}{m}}.
        \end{align*}
        Combining with (2.6), we obtain 
        \begin{align}
            \int_M u^{-\alpha t}f^{t+1}\varphi^sd\mu 
            \leq & C\left(\int_M |\nabla \varphi|^{\frac{m(\alpha-a)}{\alpha-m+1}}a^{-1-\frac{(m-1)(\alpha-a)}{\alpha-m+1}} d\mu\right)^{\frac{m-1}{m}} 
            \\
            &\cdot \left(\int_Mu^{(m-1)(a+b\alpha+1)-m\alpha t}f^{mt-(m-1)b}\varphi^{s-m} |\nabla \varphi|^m d\mu\right)^{\frac{1}{m}}. \nonumber
        \end{align}
        Set $t=\frac{m-1}{\alpha-m+1}$ such that 
        \begin{align*}
            \frac{mt-(m-1)b}{t+1}=\frac{m\alpha t-(m-1)(a+b\alpha+1) }{\alpha t}.
        \end{align*}
        Recalling that $|\nabla \varphi|=0$ on $K$ and using H$\Ddot{\mathrm{o}}$lder inequaliy, we have
        \begin{align}
            &\int_Mu^{(m-1)(a+b\alpha+1)-m\alpha t}f^{mt-(m-1)b}\varphi^{s-m} |\nabla \varphi|^m d\mu \nonumber
            \\
            =&\int_{M\backslash K}(u^{-\frac{(m-1)^2(a+1)}{\alpha-m+1}}f^{\frac{(m-1)(a+1)}{\alpha-m+1}}\varphi^{\frac{(m-1)(a+1)s}{\alpha}}) \nonumber
            \\
            &\cdot (\varphi^{\frac{\alpha-(m-1)(a+1)}{\alpha}s-m}|\nabla \varphi|^m) d\mu\nonumber
            \\
            \leq & \left(\int_{M\backslash K} u^{-\frac{(m-1)\alpha}{\alpha-m+1}}f^{\frac{\alpha}{\alpha-m+1}} \varphi^sd\mu \right)^{\frac{(m-1)(a+1)}{\alpha}}\nonumber
            \\
            &\cdot\left( \int_{M\backslash K} \varphi^{s-\frac{m\alpha}{\alpha-(m-1)(a+1)}}|\nabla \varphi|^{\frac{m\alpha}{\alpha-(m-1)(a+1)}}d\mu\right)^{\frac{\alpha-(m-1)(a+1)}{\alpha}}.
        \end{align}
        Noting that $s>\frac{m\alpha}{\alpha-(m-1)(a+1)}$ and $0\leq \varphi\leq 1$, we obtain (2.2) from (2.7) and (2.8).

        \end{proof}
        Now we are ready to prove Theorem 1.1.
        \\
        $\mathit{Proof\ of\ Theorem\ 1.1} $  First, we choose a cut-off function $h \in C^{\infty}([0,\infty))$ such that
    \begin{eqnarray}
        h|_{[0,1]}=1,\quad 0\leq h\leq 1,\quad h|_{[2,\infty)}=0, \quad |h^{'}|\leq C. \nonumber
     \end{eqnarray}
         Consider a sequence of functions $\{\varphi_i\}_{i\in \mathbb{N}}$ defined by
         \begin{align*}
             \varphi_i=i^{-1}\sum_{k=i+1}^{2i}\eta_k,
         \end{align*}
         where $\eta_k=h\left(\frac{r(x)}{2^k}\right) $ and $r(x)$ is the Riemannian distance between $x,x_0$.  It is easy to verify that the support of $\{\nabla \eta_k\}_{i\in \mathbb{N}}$ are disjoint with each other. Thus,
         \begin{align}
             |\nabla\varphi_i|^\theta=i^{-\theta}\sum_{k=i+1}^{2i}|\nabla\eta_k|^\theta\leq C i^{-\theta}\sum_{k=i+1}^{2i} 2^{-k\theta} \chi_{\{2^k\leq r(\cdot)\leq 2^{k+1}\}},
         \end{align}
         where $ \chi$ is the characteristic function. Define the integral 
         \begin{align*}
             J_i(\theta)=\int_M|\nabla \varphi_i|^\theta d\mu.
         \end{align*}
         Putting $\varphi=\varphi_i$ and $a=i^{-1}$ in (2.3), we obtain
         \begin{align}
                &\left(\int_M u^{-\frac{(m-1)\alpha}{\alpha-m+1}}f^{\frac{\alpha}{\alpha-m+1}}\varphi^sd\mu\right)^{1-\frac{(m-1)(a+1)}{m\alpha}} \nonumber
                \\
                \leq & C i^{\frac{m-1}{m}+\frac{(m-1)^2(\alpha-i^{-1})}{m(\alpha-m+1)} }\left(J_i\left(\frac{m(\alpha-i^{-1})}{\alpha-m+1}\right)\right)^{\frac{m-1}{m}}
                \nonumber\\
                &\cdot \left(J_i\left(\frac{m\alpha}{\alpha-(m-1)(a+1)}\right) \right)^{\frac{\alpha-(m-1)(i^{-1}+1)}{m\alpha}}.
            \end{align}
           From the estimate of $Vol B_{x_0}(R)$, for $i$ large enough, we have
         \begin{align}
             J_i(\theta)=&\int_M|\nabla \varphi_i|^\theta d\mu=\int_{supp|\nabla \varphi_i|}|\nabla \varphi_i|^\theta d\mu\leq Ci^{-\theta}\sum_{k=i+1}^{2i} 2^{-k\theta} V(2^{k+1}) \nonumber
             \\
             \leq &C i^{-\theta}\sum_{k=i+1}^{2i}2^{k(p-\theta)}k^q\leq C i^{q-\theta}\sum_{k=i+1}^{2i}2^{k(p-\theta)}.
         \end{align}
         Using (2.11) with $\theta=\frac{m(\alpha-i^{-1})}{\alpha-m+1}$, we have
         \begin{align}
             J_i\left( \frac{m(\alpha-i^{-1})}{\alpha-m+1}\right)\leq C i^{q-\theta}\sum_{k=i+1}^{2i}2^{\frac{km}{i(\alpha-m+1)}}\leq Ci^{\frac{(1-m)\alpha}{\alpha-m+1}}.
         \end{align}
         Similarly, for $\theta=\frac{m\alpha}{\alpha-(m-1)(i^{-1}+1)}$, we obtain
         \begin{align}
             J_i\left( \frac{m\alpha}{\alpha-(m-1)(i^{-1}+1)}\right)\leq  Ci^{\frac{(1-m)\alpha}{\alpha-m+1}}.
         \end{align}
         Substituting (2.12) and (2.13) into (2.10), we conclude that 
         \begin{align}
                &\left(\int_M u^{-\frac{(m-1)\alpha}{\alpha-m+1}}f^{\frac{\alpha}{\alpha-m+1}}\varphi^sd\mu\right)^{1-\frac{(m-1)(a+1)}{m\alpha}} \nonumber
                \\
                \leq & C i^{\frac{m-1}{m}+\frac{(m-1)^2(\alpha-i^{-1})}{m(\alpha-m+1)} }\cdot (i^{\frac{(1-m)\alpha}{\alpha-m+1}})^{\frac{m-1}{m}+\frac{\alpha-(m-1)(i^{-1}+1)}{m\alpha}}\leq C ,\nonumber
            \end{align}
            that is,
         \begin{align}
             \int_{B_{x_0}(2^{i+1})} u^{-\frac{(m-1)\alpha}{\alpha-m+1}}f^{\frac{\alpha}{\alpha-m+1}}d\mu \leq C.
         \end{align}
         Letting $i \rightarrow \infty$, by Fatou’s lemma, we obtain
         \begin{align}
             \int_M u^{-\frac{(m-1)\alpha}{\alpha-m+1}}f^{\frac{\alpha}{\alpha-m+1}}d\mu \leq C.
         \end{align}
         Combing with (2.2), we have 
         \begin{align*}
             \int_{ B_{x_0}(2^{i+1})} u^{-\frac{(m-1)\alpha}{\alpha-m+1}}f^{\frac{\alpha}{\alpha-m+1}}d\mu \leq  C \left(\int_{M\backslash B_{x_0}(2^{i+1})} u^{-\frac{(m-1)\alpha}{\alpha-m+1}}f^{\frac{\alpha}{\alpha-m+1}}d\mu \right)^{\frac{(m-1)(1+i^{-1})}{m\alpha}}.
         \end{align*}
         Letting $i \rightarrow \infty$ again, we obtain
         \begin{align*}
             \int_{M} u^{-\frac{(m-1)\alpha}{\alpha-m+1}}f^{\frac{\alpha}{\alpha-m+1}}d\mu=0,
         \end{align*}
         which contradicts the assumption that $u>0$ and $f(u)>0$.
         {\qed}

         \section{ Acknowledgments}
        The author would like to thank Professor Yuxin Dong and Professor Xiaohua Zhu for their continued support and encouragement.

\bibliographystyle{siam}
\bibliography{plaplace}

~\\
  Biqiang Zhao
  \\
  $Beijing\ International\ Center\ for$
  \\
  $Mathematical\ Research $
\\
  $ Peking\ University$
\\
   $Beijing\ 100871 ,$ $P.R.\ China $

\end{document}